\documentclass[a4paper,11pt]{article}
\usepackage[latin1]{inputenc}
 \usepackage[T1]{fontenc}
  \usepackage{amsmath,amsthm,amssymb}
 \usepackage[all]{xy}

\def \N{{\mathbb N}}

\def \R{{\mathbb R}}

\def \1{{\mathbb 1}}

\newtheorem{Exemp}{Examples}
\newtheorem{Thm}{Theorem}
\newtheorem{Prop}{Proposition}
\newtheorem{Def}{Definition}

\newtheorem{Lem}{Lemma}

\newtheorem{Def Nota}{Definitions and notations} 
\newtheorem{Cor}{Corollary}

\font\ninerm=cmr9
\long\outer\def\abstract#1{\bigskip\vbox{\noindent\ninerm
\baselineskip=10pt#1}\nobreak\bigskip}
\newcount\numero
\def\exo#1{\advance\numero by 1\bigskip
{\noindent\tenbf #1\the\numero. }}
\def\frac#1#2{{#1\over #2}}
\title{A convex extension of lower semicontinuous functions defined on normal Hausdorff space.}   
\author{Mohammed Bachir}
\begin{document}
\maketitle
\begin{center} {\it Laboratoire SAMM 4543, Universit\'e Paris 1 Panth\'eon-Sorbonne, Centre P.M.F. 90 rue Tolbiac 75634 Paris cedex 13}
\end{center}
\begin{center} 
{\it Email : Mohammed.Bachir@univ-paris1.fr}
\end{center}
\noindent\textbf{Abstract.} We prove that, any problem of minimization of proper lower semicontinuous function defined on a normal Hausdorff space, is canonically equivalent to a problem of minimization of a proper weak$^*$ lower semicontinuous convex function defined on a weak$^*$ convex compact subset of some dual Banach space. We estalish the existence of an bijective operator between  the two classes of functions which preserves the problems of minimization.
%

\noindent {\bf Msc: } 47N10, 46N10, 46E15

\tableofcontents

\section{Introduction.}




Let $(X,\tau)$ be a completely regular and Hausdorff space. By $(C_b(X),\|.\|)$ we denote the Banach space of all real-valued, bounded and continuous functions, equipped with the sup-norm. The continuous Dirac map is defined by 

\begin{eqnarray*}
\Delta : (X,\tau) &\longrightarrow& (\Delta(X),w^*)\subset (B_{(C_b(X))^*},w^*)\\
         x &\mapsto & \Delta(x)
\end{eqnarray*}
where, $(B_{(C_b(X))^*},w^*)$ is the closed unit ball of the dual space $(C_b(X))^*$ equipped with the weak$^*$ topology and $\Delta(x) : C_b(X)\longrightarrow \R$ is the linear continuous map defined by $\Delta(x)(\varphi)=\varphi(x)$ for all $x\in X$ and all $\varphi\in C_b(X)$. It is well know that $\Delta$ is a homeomorphism from $(X,\tau)$ onto $(\Delta(X),w^*)$, it is also well know that the set $\overline{\Delta(X)}^{w^*}$ coincides (up to homeomorphism) with the Stone-\v{C}ech compactification $\beta X$. The Stone-\v{C}ech compactification $\beta X$, has the property that every continuous function $\varphi$ from  $X$ into a compact Hausdorff space $K$, has a unique extension to a continuous function $\beta \varphi$ from $\beta X$ into $K$. For more informations about the Stone-\v{C}ech compactification, we refer to \cite{W}. 

\vskip5mm

We are interested in this paper, on a canonical extension of bounded from below lower semicontinuous function $f$ from $X$ into $\R$ to a bounded from below extended convex weak$^*$ lower semicontinuous function $\mathcal{T}(f)$ defined on the convex weak$^*$ compact set $\overline{\textnormal{co}}^{w^*}(\Delta(X))$ (the weak$^*$ closed convex hull of $\Delta(X)$, which can be seen as the convexification of the Stone-\v{C}ech compact $\beta X$). To prove the existence of a canonical extension of any bounded from below lower semicontinuous function $f$ from $X$ into $\R$, we need to assume that $X$ is a normal Hausdorff space. Recall that a topological space $X$ is a normal space if, given any disjoint closed sets $E$ and $F$, there are neighbourhoods $U$ of $E$ and $V$ of $F$ that are also disjoint. The tools that will be used for this purpose, are based on a non convex analogue to Fenchel duality introduced in \cite{Ba}.

\vskip5mm
If $A$ is a subset of $C_b(X)$, we denote by $\sigma_A$ the support function, defined on the dual space by 
\begin{eqnarray*}
\forall Q\in (C_b(X))^*;\hspace{2mm} \sigma_A(Q):=\sup_{\varphi\in A} \langle Q, \varphi\rangle.
\end{eqnarray*}

\begin{Def} A nonempty subset $A$ of $C_b(X)$ is said to be a $\Delta$-set, if and only if, for all $x\in X$ there exist a real number $\lambda_x\in \R$, such that 
\begin{eqnarray*} 
A &=& \cap_{x\in X} \lbrace \varphi \in C_b(X)/ \langle \Delta(x),\varphi\rangle\leq \lambda_x\rbrace\\
  &:=& \cap_{x\in X} \lbrace \varphi \in C_b(X)/ \varphi(x)\leq \lambda_x\rbrace.
\end{eqnarray*} 
\end{Def} 
A $\Delta$-set of $C_b(X)$ is necessarily a closed convex subset of $C_b(X)$. We denote by $SCI(X)$ the set of all real-valued functions $f: X \longrightarrow \R$, bounded from below and lower semicontinuous and by $\Gamma(\overline{\textnormal{co}}^{w^*}(\Delta(X))$ we denote the set of all convex weak$^*$ lower semicontinuous functions that are the restrictions to $\overline{\textnormal{co}}^{w^*}(\Delta(X))$ of the functions $\sigma_A$, where $A$ is a $\Delta$-set of $C_b(X)$:
\begin{eqnarray*}
\Gamma(\overline{\textnormal{co}}^{w^*}(\Delta(X))):=\lbrace (\sigma_A)_{|\overline{\textnormal{co}}^{w^*}(\Delta(X))}/ A \textnormal{ is a } \Delta\textnormal{-set of } C_b(X)\rbrace. 
\end{eqnarray*}

It is easy to see that $SCI(X)$ is a convex cone. The aim of this paper is to prove the following result.

\begin{Thm} \label{Iso-cone} Let $X$ be a normal Hausdorff space. Then, the set $\Gamma(\overline{\textnormal{co}}^{w^*}(\Delta(X)))$ is a convex cone and there exists an isomorphism of convex cone $\mathcal{T}: SCI(X)\longrightarrow \Gamma(\overline{\textnormal{co}}^{w^*}(\Delta(X)))$ i.e. $\mathcal{T}$ is bijective and for all $f, g \in SCI(X)$ and all $\alpha,\beta \in \R^+$, we have $$\mathcal{T}(\alpha f+\beta g)=\alpha\mathcal{T}(f)+ \beta \mathcal{T}(g).$$ This isomorphism satisfies also the following properties.

$(1)$ For all $f \in SCI(X)$, $\mathcal{T}(f)\circ \Delta =f$, this means that $\mathcal{T}(f)$ is a convex weak$^*$ lower semicontinuous extension of $f$ to $\overline{\textnormal{co}}^{w^*}(\Delta(X))$, up to the identification between $X$ and $\Delta(X)$.

$(2)$ The convex cone isomorphime $\mathcal{T}$ is isotone i.e. for all $f, g \in SCI(X)$, we have that 
$$ f \leq g \Longleftrightarrow   \mathcal{T}(f) \leq \mathcal{T}(g)$$

$(3)$ For all bounded continuous function $\varphi\in C_b(X)$ and all $Q\in \overline{\textnormal{co}}^{w^*}(\Delta(X))$, we have $$\mathcal{T}(\varphi)(Q)=\langle Q, \varphi \rangle :=\hat{\varphi}(Q).$$
This means that for all $\varphi\in C_b(X)$, the function $\mathcal{T}(\varphi)$ is affine and weak$^*$ continuous on $\overline{\textnormal{co}}^{w^*}(\Delta(X))$.

$(4)$ For all $f \in SCI(X)$, we have $$\inf_X f = \min_{\overline{\Delta(X)}^{w^*}}\mathcal{T}(f)=\min_{\overline{\textnormal{co}}^{w^*}(\Delta(X))}\mathcal{T}(f).$$

\end{Thm}
\vskip5mm

In the last part of the above theorem, the weak$^*$ convex lower semicontinuous function $\mathcal{T}(f)$ always attains its minimum on $\overline{\Delta(X)}^{w^*}$ since it is a compact set, this is not the case in general for $f$. From the parts $(1)$ and  $(4)$ of the above theorem, we get that a point $x\in X$ is a minimum for a lower semicontinuous function $f$ iff $\Delta(x)$ is a minimum for the convex weak$^*$ lower semicontinuous function $\mathcal{T}(f)$. More generally, if $Y\subset C_b(X)$ is a class of perturbations, then from the fact that $\mathcal{T}$ is a cone isomorphism we have that $\mathcal{T}(f+\varphi)=\mathcal{T}(f)+\mathcal{T}(\varphi)$. Using part $(3)$, we get that $\mathcal{T}(f+\varphi)=\mathcal{T}(f)+ \hat{\varphi}$. Thus, from parts $(1)$ and $(4)$, we have that for each $\varphi \in Y$, $f+\varphi$ has a minimum at some point $x\in X$ iff $\mathcal{T}(f)+ \hat{\varphi}$ has a minimum at $\Delta(x)$. This shows that a non linear variational principle for lower semicontinuous functions $f$ defined on normal Hausdorff space is equivalent to a linear variational principle for convex weak$^*$ lower semicontinuous functions $\mathcal{T}(f)$ defined on a convex weak$^*$ compact subset of some dual Banach space. The just mentioned remark can be seen as a linearization of the Deville-Godefroy-Zizler variational principle \cite{DGZ}, \cite{DGZbook}, \cite{DR}.
\vskip5mm
This main result will be proved at the end of this note when preliminary results are established in the next sections. These preliminary results are of interest in themselves.
\section{Duality and linearization results.}
Let $X$ be a topological space and $C(X)$ the space of all real-valued continuous functions on $X$. Let $Y$ be a non empty subset of $C(X)$. Let $f: X\longrightarrow \R\cup \lbrace +\infty\rbrace$ be a function such that $dom(f):=\lbrace x\in X/ f(x)< +\infty \rbrace \neq \emptyset$. By $A_Y(f)$ we denote the set  
$$A_Y(f):=\lbrace \varphi\in Y/ \varphi \leq f\rbrace.$$
We introduced in \cite{Ba} a non convex analogue to Fenchel duality, where relations between well-posedness and differentiability was established. We recall that the conjugate of $f$ depending on the class of functions $Y$ is defined as follows:  for all $\varphi \in Y$,
 $$ f^{\times} (\varphi):=\sup_{x\in X} \lbrace \varphi(x) - f(x)\rbrace.$$ 
Note that $dom(f^{\times}) \neq \emptyset$ if and only if, there exists a real number $c\in \R$ such that $A_Y(f+c)\neq \emptyset$, this condition is satisfed in particular if $f$ is bounded from below.  
The second conjugate of $f$ is defined on $X$ as follows: for all $x\in X$
 $$ f^{\times\times} (x):=\sup_{\varphi\in Y} \lbrace \varphi(x) - f^{\times}(\varphi)\rbrace.$$ 
Note that we always have that $f^{\times\times} \leq f$. 

\subsection{Conjugacy of lower semicontinuous function.}

If the class $Y$ is a vector subspace of $C(X)$, then $f^{\times}$ is a convex function on $Y$ as a supremum of affine maps on $Y$. If moreover, $Y$ is a vector subspace of $C_b(X)$ (the space of all real-valued bounded and continuous functions), then it is easy to see that $f^{\times}$ is a convex and $1$-Lipschitz map for the norm $\|.\|_{\infty}$, for every bounded from below function$f$. But in general $f^{\times\times}$ is not convex on $X$ even if $X$ is a vector space. Actually, we get in Theorem \ref{dual1} that under general conditions on the pair $(X,Y)$, we have that $f$ is lower semicontinuous function if and only if, $f^{\times\times}=f$. This result was initially obtained in [Theorem 1, \cite{Ba}], when $X$ is a metric space and $Y$ is a subspace of $C_b(X)$ containing a bump function (a bounded function on $X$ with a nonempty support). In fact, this result is true in a more general setting with essentially the same proof. Since we need to use this result in a general topological space, we give its proof in its general setting. 

\begin{Def} Let $X$ be a topological space and $Y$ be a nonempty subset of $C(X)$. We say that the pair $(X,Y)$ satisfies the property $(H)$ if and only if, for each $x\in X$ and each open neighborhood $U$ of $x$, there exists $\sigma  : X\longrightarrow [0,1]$ such that $\sigma \in Y$, $\sigma (x)=1$ and $\sigma (y)=0$ for all $y \in X\setminus U$.
\end{Def}


\begin{Exemp} \label{Exemp1}
We have the following examples.

$(1)$ If $X$ is a normal Hausdorff space, then $(X,C_b(X))$ satisfies $(H)$ tanks to the Urysohn's lemma.

$(2)$ If $(X,d)$ is a metric space, then it is easy to see that $(X,Lip_b(X))$ satisfies $(H)$, where $Lip_b(X)$ denotes the space of all real-valued  bounded and Lipschitz map on $X$.

$(3)$ If $X$ is a Banach space having a $k$-times ($k \in \N \cup \lbrace +\infty\rbrace $) continuously differentiable and uniformly bounded bump function (see \cite{DGZbook} for examples of Banach spaces having this property), then of course $(X,C_b^k(X))$ satisfies $(H)$, where  $C_b^k(X)$ denotes the space of all $k$-continuously differentiable and uniformly bounded functions on $X$.
\end{Exemp}

\begin{Thm} \label{dual1} Let $X$ be a topological space and $Y$ be a cone included in $C(X)$. Suppose that the pair $(X,Y)$ satisfies the hypothesis $(H)$. Let $f: X \longrightarrow \R \cup \lbrace +\infty \rbrace$ be a function with $dom(f)\neq \emptyset$ and such that $A_Y(f)\neq \emptyset$. Then, $f$ is lower semi-contiuous if and only if $f^{\times\times}= f$.
\end{Thm}

\begin{proof} The "only if" part follows from the definition of $f^{\times\times}$ and the fact that, the supremum of continuous functions is a lower semicontinuous function. Let us prove the "if" part. Since $A_Y(f)\neq \emptyset$, there exists $\varphi_0\in Y$ such that $\varphi_0\leq f$. Set $g:=f-\varphi_0 \geq 0$ and let us proof that $g^{\times\times}=g$. Indeed, let $x \in X$ and take any real number $a$ such that $a<g (x)$. We prove that $a\leq g^{\times\times}(x)$. Indeed, since $g$ is lower semicontinuous, there exists an open neighborhood $U$ of $x$ such that $a< g(y)$ for all $y \in U$. From the hypothesis $(H)$, there exists a continuous function $\sigma  : X\longrightarrow [0,1]$ such that $\sigma \in Y$, $\sigma (x)=1$ and $\sigma (y)=0$ for all $y \in X\setminus U$. We define $\varphi_x(y):=(g(x)-\inf_X g)\sigma (y)$ for all $y\in X$. Clearly, $\varphi_x\in Y$ for all $x\in X$, since $Y$ is a cone and $\sigma \in Y$. By separately examining the case where $y\in U$ and $y\in X\setminus U$, we can easily verify that $$\varphi_x(y)- g(y)< \varphi_x(x)-a.$$
Taking the supremum over $y\in X$, we obtain that $g^{\times}(\varphi_x) \leq \varphi_x(x)-a$. Thus,we obtain $a\leq \varphi_x(x)-g^{\times}(\varphi_x)\leq g^{\times\times}(x)$. This proves that $a\leq g^{\times\times}(x)$ for all $a< g(x)$. Hence $g(x)\leq g^{\times\times}(x)$ for all $x\in X$. On the other hand, it follows from the definition of the second conjugacy that $g^{\times\times}(x)\leq g(x)$ for all $x\in X$. Thus, $g^{\times\times}(x)= g(x)$ for all $x\in X$. Now, replacing $g$ by $f-\varphi_0$, we get that $g^{\times\times}=f^{\times\times}-\varphi_0$. So we obtain that $f^{\times\times}-\varphi_0=f-\varphi_0$ i.e. $f^{\times\times}=f$.

\end{proof}
 
\begin{Cor} \label{cor1} Let $X$ be a topological space and $Y$ be a cone including in $C(X)$ containing the constants. Suppose that the pair $(X,Y)$ satisfies $(H)$. Let $f: X \longrightarrow \R \cup \lbrace +\infty \rbrace$ be a fuction with $dom(f)\neq \emptyset$ and such that $A_Y(f)\neq \emptyset$. Then, $f$ is lower semicontinuous if and only if, $f(x)=\sup_{\varphi\in A_Y(f)} \varphi(x)$ for all $x\in X$. In other words, $f$ is the supremum of functions from $Y$ that minors $f$ from below.
\end{Cor}
\begin{proof} Suppose that $f$ is lower semicontinuous. On one hand, we have $\sup_{\varphi\in A_Y(f)} \varphi(x) \leq f(x)$ for all $x\in X$. On the other hand, we known from Theorem \ref{dual1} that $f(x)= f^{\times\times}(x)=\sup_{\psi\in Y}\lbrace \psi(x)-f^{\times}(\psi)\rbrace \leq \sup_{\varphi\in A_Y(f)} \varphi(x)$, since $\psi(.)-f^{\times}(\psi)\leq f$ and $\psi(.)-f^{\times}(\psi)\in Y$ for each $\psi\in Y$. Now, if $f(x)=\sup_{\varphi\in A_Y(f)} \varphi(x)$ for all $x\in X$, then $f$ is lower semicontinuous as supremum of continuous functions.
\end{proof}
\begin{Exemp} From Corollary \ref{cor1} and Example 1., we get that

$(1)$ if $X$ is a normal Hausdorff space, then each lower semicontinuous function $f$ with a nonempty domain and bounded from below by a continuous function, is the supremum of the continuous functions that minors $f$ from below.

$(2)$ if $X$ is a metric space, then each lower semicontinuous function $f$ with a nonempty domain and bounded from below by a Lipschitz continuous function, is the supremum of the Lipschitz continuous functions that minors $f$ from below.

$(3)$ if $X$ is a Banach space having a $k$-times ($k \in \N \cup \lbrace +\infty\rbrace $) continuously differentiable bump function, then each lower semicontinuous function $f$ with a nonempty domain and bounded from below by a $k$-times continuously differentiable function, is the supremum of the $k$-times continuously differentiable function that minors $f$ from below.
\end{Exemp}
\vskip5mm

\subsection{A convex extention of lower semicontinuous function.}

Let $X$ be a  normal Hausdorff space and $(Y,\|.\|)$ is a Banach space included in $C_b(X)$ such that $\|.\|\geq \|.\|_{\infty}$ and $(X,Y)$ satisfies the hypothesis $(H)$. Clearly, this conditions implies the following properties:

$(1)$ $\|.\|\geq \|.\|_{\infty}$

$(2)$ $Y$ separate the points of $X$ 

$(3)$ for each $x\in X$, there exists $\varphi \in Y$ such that $\varphi(x)=1$.

\vskip5mm

For each $x\in X$, we denote by $\delta_x$ the Dirac evaluation defined by $\delta_x(\varphi)=\varphi(x)$ for all $\varphi\in Y$. The continuity of the linear map $\delta_x$ is guaranteed by the condition $(1)$ above. Clearly, $\delta_x \in B_{Y^*}$ (the unit ball of the topological dual space of $Y$) for all $x\in X$. We define $\Delta_Y: (X,\tau) \longrightarrow (Y^*, w^*)$ by $\Delta_Y(x)=\delta_x$ for all $x\in X$. The injectivity of $\Delta_Y$ is guaranteed by the condition $(2)$. We need the following proposition.

\begin{Prop} \label{Homeo} Let $(X,\tau)$ be a  normal Hausdorff space and $(Y,\|.\|)$ is a Banach space included in $C_b(X)$ such that $\|.\|\geq \|.\|_{\infty}$ and $(X,Y)$ satisfies the hypothesis $(H)$. Then, the map $\Delta_Y$ is an homeomorphism from $(X,\tau)$ onto $(\Delta_Y(X),w^*)$.
\end{Prop}

\begin{proof} Let $x, y \in X$ such that $x\neq y$. Since $Y$ separate the points of $X$, there exists $\sigma  \in Y$ such that $\sigma (x)\neq \sigma (y)$ and so $\Delta_Y$ is one-to-one. Clearly, $\Delta_Y$ is $\tau$-$w^*$-continuous, since $Y \subset C_b(X)$. Let us prove that $\Delta$ is open. Let $a\in X$ and $U$ an open set such that $a\in U$. We prove that there exists an open set $V$ of $\Delta_Y(X)$ such that $\Delta_Y(a)\in V$ and $V\subset \Delta_Y(U)$. Indeed, by the hypothesis $(H)$, there exists $\sigma  : X\longrightarrow [0,1]$ such that $\sigma \in Y$, $\sigma (a)=1$ and $\sigma (y)=0$ for all $y \in X\setminus U$. Set $W:=\lbrace y^* \in Y^*/ y^*(\sigma ):= \hat{\sigma }(y^*)>0 \rbrace$. We have that $W=\hat{\sigma }^{-1}(]0, +\infty[)$ and so it is a weak-star open subset of $Y^*$, moreover $\Delta_Y(a)\in W$. We set $V:=W\cap \Delta_Y(X)$. 

\end{proof}

The set $\overline{\Delta_Y(X)}^{w^*}\subset B_{Y^*}$ is a weak-star compact subset of $Y^*$ by the Banach-Alaoglu theorem. Note that when we take $Y=C_b(X)$, then the set $\overline{\Delta_Y(X)}^{w^*}$ coincides (up to a homeomorphism) with the Stone-\v{C}ech compactification $\beta X$.

\vskip5mm

For all $\varphi\in Y$ and all $Q\in Y^*$, we will use, according to the situations, the following equivalent notations 
$$\langle Q, \varphi\rangle:=Q(\varphi):=\hat{\varphi}(Q).$$ 
Now, given a bounded from below function $f$ defined on $X$, we denote by $\mathcal{F}(f)$ the Fenchel transform of the conjugacy $f^{\times}$ defined on the dual space $Y^*$ by :
$$\mathcal{F}(f)(Q):=(f^{\times})^*(Q):=\sup_{\varphi\in Y}\lbrace \langle Q, \varphi \rangle -f^{\times}(\varphi)\rbrace, \hspace{3mm} \forall Q\in Y^*.$$

We know that $\mathcal{F}(f)$ is convex and weak-star lower semicontinuous as Fenchel transform of the convex $1$-Lipschitz function $f^{\times}$ on $Y$ ([Proposition 1., \cite{Ba}]). 

\vskip5mm 
In the following lemma, we study some properties of the operator $\mathcal{F}$ which will be used in the next sections. 
\begin{Lem} \label{Prop1} Let $X$ be a normal Hausdorff space and $(Y,\|.\|)$ is a Banach space included in $C_b(X)$ such that $\|.\|\geq \|.\|_{\infty}$ and that the pair $(X,Y)$ satisfies $(H)$. Let $f: X\rightarrow \R \cup\left\{ + \infty \right\}$ be a bounded from below and lower semicontinuous fonction with $dom(f)\neq \emptyset$. Then, the following assertions holds.

$(1)$ The function $\mathcal{F}(f)$ is convex weak-star semicontinuous. We have $dom(\mathcal{F}(f))\subset \overline{\textnormal{co}}^{w^*}(\Delta_Y(X))$ and $f=\mathcal{F}(f)\circ \Delta_Y$. In other words, the following diagram commutes 

\[
  \xymatrix{
     X \ar[r]^{\Delta_Y} \ar[rd]_{f}  &  Y^* \ar[d]^{\mathcal{F}(f)} \\
      & \R\cup \lbrace+\infty\rbrace } 
 \]

$(2)$ When $f\equiv c$, where $c\in \R$, we have $\mathcal{F}(c)=c+i_{\overline{\textnormal{co}}^{w^*}(\Delta_Y(X))}$, where $i_{\overline{\textnormal{co}}^{w^*}(\Delta_Y(X))}$ is the indicator function of $\overline{\textnormal{co}}^{w^*}(\Delta_Y(X))$ which is equal to $0$ on $\overline{\textnormal{co}}^{w^*}(\Delta_Y(X))$ and $+\infty$ otherwise.

$(3)$ For all $\varphi\in Y$ we have $\mathcal{F}{(f-\varphi)}=\mathcal{F}(f)-\hat{\varphi}$. In particular we have $\mathcal{F}(\varphi)=\hat{\varphi}+i_{\overline{\textnormal{co}}^{w^*}(\Delta_Y(X))}$ for all $\varphi \in Y$.

$(4)$ We have the conservation of the infinimums : $$\inf_{x\in X} f(x)=\min_{ Q \in \overline{\Delta_Y(X)}^{w^*}} \mathcal{F}(f)(Q)=\min_{ Q \in \overline{\textnormal{co}}^{w^*}(\Delta_Y(X))} \mathcal{F}(f)(Q).$$

$(5)$ If $(x_n)_n$ is a sequence that minimize the function $f$ on $X$, then $(\delta_{x_n})_n$ is a sequence that minimize $\mathcal{F}(f)$ on $Y^*$.

\end{Lem}
\begin{proof}
 $(1)$ From the definition of $f^{\times}$ we have $f^{\times}(\varphi)\leq \sup_{x\in X} \varphi(x)-\inf_{x\in X} f(x)=\sup_{x\in X} \varphi(x) + f^{\times}(0)$ for all $\varphi \in Y$. Thus, for all $\varphi \in Y$ and all $Q\in (Y)^*$ we have  $$Q(\varphi) - f^{\times}(\varphi)\geq Q(\varphi) -\sup_{x\in X} \varphi(x) - f^{\times}(0).$$ Let $Q\notin \overline{\textnormal{co}}^{w^*}(\Delta_Y(X))$, by the Hahn-Banach theorem, there exists $\varphi_0 \in Y$ such that $Q(\varphi_0)=\hat{\varphi}_0(Q) > 1\geq \sup_{S\in \overline{\textnormal{co}}^{w^*}(\Delta_Y(X))} \hat{\varphi}_0(S)$. On the other hand, since $\Delta_Y(X)\subset \overline{\textnormal{co}}^{w^*}(\Delta_Y(X))$, we have that $\sup_{S\in \overline{\textnormal{co}}^{w^*}(\Delta_Y(X))} \hat{\varphi}_0(S)\geq \sup_{x\in X} \varphi_0(x)$. Thus $$Q(\varphi_0) - \sup_{x\in X} \varphi_0(x)\geq Q(\varphi_0) -\sup_{S\in \overline{\textnormal{co}}^{w^*}(\Delta_Y(X))} \hat{\varphi}_0(S) >0.$$
Hence, for all $\lambda \in \R^*_+$, we have $$Q(\lambda\varphi_0) - \sup_{x\in X} \lambda\varphi_0(x)\geq \lambda\left(Q(\varphi_0) -\sup_{\overline{\textnormal{co}}^{w^*}(\Delta_Y(X))} (\varphi_0)\right)\stackrel{\lambda\rightarrow +\infty}{\longrightarrow} +\infty.$$
This implies that $\mathcal{F}(f)(Q)=+\infty$ whenever $Q\notin \overline{\textnormal{co}}^{w^*}(\Delta_Y(X))$. Thus $dom(\mathcal{F}(f)) \subset \overline{\textnormal{co}}^{w^*}(\Delta_Y(X))$. Now, the fact that $f=\mathcal{F}(f)\circ \delta$ follows from Theorem \ref{dual1}.\\
\item $(2)$ By definition we have :
\begin{eqnarray}
	\mathcal{F}(c)(Q)=\sup_{\varphi \in Y}\left\{Q(\varphi)-c^{\times}(\varphi)\right\}
	=\sup_{\varphi \in Y}\left\{Q(\varphi)-\sup_{x\in X}\varphi(x)\right\} +c\nonumber
\end{eqnarray}
We deduce from the above equality that $\mathcal{F}(c)(Q)\geq c$ for all $Q\in (Y)^*$. On the other hand, let $Q\in \overline{\textnormal{co}}^{w^*}(\Delta_Y(X))$, there exists a net $(Q_\alpha)_{\alpha}\subset \textnormal{conv}(\Delta_Y(X))$ such that $Q_\alpha \stackrel{w^*}{\rightarrow} Q$. For each $\alpha$, there exists $(\lambda^\alpha_i)_{i=1}^{n}\subset \R_+$ and $(x^\alpha_i)_{i=1}^n\subset X$ such that $\sum_{i=1}^n \lambda^\alpha_i=1$ and $Q_\alpha= \sum_{i=1}^n \lambda^\alpha_i \delta_{x^\alpha_i}$. We then have  $Q_\alpha(\varphi)=\sum_{i=1}^n \lambda^\alpha_i \varphi(x^\alpha_i)\leq \sup_{x \in X}\varphi(x)$ for all $\varphi \in Y$. By taking the weak-star limit, we obtain that $Q(\varphi)\leq \sup_{x \in X}\varphi(x)$ for all $\varphi \in Y$. Thus, using the above formula, we get $\mathcal{F}(c)(Q)\leq c$ for all $Q\in \overline{\textnormal{co}}^{w^*}(\Delta_Y(X))$ and so we have that $\mathcal{F}(c)(Q)= c$ for $Q\in \overline{\textnormal{co}}^{w^*}(\Delta_Y(X))$. From the part $(1)$, since $dom(\mathcal{F}(c))\subset \overline{\textnormal{co}}^{w^*}(\Delta_Y(X))$, we conclude that $\mathcal{F}(c)=c+i_{\overline{\textnormal{co}}^{w^*}(\Delta_Y(X))}$.\\
\item $(3)$ Let $\varphi\in Y$. By definition, we have for all $Q\in (Y)^*$ that $$\mathcal{F}{(f-\varphi)}(Q):=\sup_{\psi\in Y}\left\{Q(\psi) - (f-\varphi)^{\times} (\psi) \right\}.$$ Also by definition we have $(f-\varphi)^{\times} (\psi)=(f)^{\times} (\varphi+ \psi)$ for all $\varphi, \psi\in Y$. By a change of variable we obtain for all $Q\in (Y)^*$:
\begin{eqnarray}
 \mathcal{F}{(f-\varphi)}(Q):&=&\sup_{\psi\in Y}\left\{Q(\psi-\varphi) - (f)^{\times} (\psi) \right\}\nonumber\\ 	
&=& \sup_{\psi\in Y}\left\{Q(\psi) - (f)^{\times} (\psi) \right\}-Q(\varphi)\nonumber\\
&=&\mathcal{F}(f)(Q)-\hat{\varphi}(Q)\nonumber
\end{eqnarray}
Thus $\mathcal{F}{(f-\varphi)}=\mathcal{F}(f)-\hat{\varphi}$. In particular, when we take $f\equiv 0$ and by using the part $(2)$, we obtain that $\mathcal{F}(\varphi)=\hat{\varphi}+i_{\overline{\textnormal{co}}^{w^*}(\Delta_Y(X))}$, for all $\varphi\in Y$.\\

\item $(4)$ Let us prove that 
$$\inf_{x\in X} f(x)=\min_{ Q \in \overline{\Delta_Y(X)}^{w^*}} \mathcal{F}(f)(Q)=\min_{ Q \in \overline{\textnormal{co}}^{w^*}(\Delta_Y(X))} \mathcal{F}(f)(Q).$$
First, note that since $\mathcal{F}(f)$ is weak$^*$ lower semicontinuous and $\overline{\Delta_Y(X)}^{w^*}$ and $\overline{\textnormal{co}}^{w^*}(\Delta_Y(X))$ are weak$^*$ compact, then $\mathcal{F}(f)$ attains its minimum on these sets. Now, since $\Delta(X)\subset \overline{\Delta_Y(X)}^{w^*}\subset \overline{\textnormal{co}}^{w^*}(\Delta_Y(X))$ then, using the part $(1)$  we have
\begin{eqnarray} \label{inf}
\min_{ Q \in \overline{\textnormal{co}}^{w^*}(\Delta_Y(X))} \mathcal{F}(f)(Q)&\leq& \min_{ Q \in \overline{\Delta_Y(X)}^{w^*} } \mathcal{F}(f)(Q)\\
                                                                     &\leq& \inf_{ Q \in \Delta(X) } \mathcal{F}(f)(Q)\\
                                                                     &=& \inf_{x\in X} f(x).
\end{eqnarray}
On the other hand, it follows from the definition that $\mathcal{F}(f)(Q)\geq -f^{\times}(0):=\inf_{x\in X} f(x)$ for all $Q\in (Y)^{*}$. Thus, $\inf_{ Q \in \overline{\textnormal{co}}^{w^*}(\Delta_Y(X))} \mathcal{F}(f)(Q)\geq \inf_{x\in X} f(x).$  Hence, $$\inf_{x\in X} f(x)=\inf_{ Q \in \overline{\textnormal{co}}^{w^*}(\Delta_Y(X))} \mathcal{F}(f)(Q)=\min_{ Q \in \overline{\textnormal{co}}^{w^*}(\Delta_Y(X))} \mathcal{F}(f)(Q).$$ Now, since  $\Delta(X)\subset \overline{\Delta_Y(X)}^{w^*}\subset \overline{\textnormal{co}}^{w^*}(\Delta_Y(X))$, using the part $(1)$, we get the conclusion.\\
\item $(5)$ Let  $(x_n)_n$ be a sequence that minimize $f$ on $X$. Since $f(x)=\mathcal{F}(f)(\delta_x)$ for all $x\in X$ and $\inf_{x\in X} f(x)=\min_{ Q \in \overline{\textnormal{co}}^{w^*}(\Delta_Y(X))} \mathcal{F}(f)(Q)$ then the seqence $(\delta_{x_n})_n$ minimize $\mathcal{F}(f)$ on $\overline{\textnormal{co}}^{w^*}(\Delta_Y(X))$.
\end{proof}

Now, we prove the following particular case in the compact framework.

\begin{Cor} Let $(X, \tau)$ be a compact Hausdorff space. Then, there exists a weak$^*$ compact subspace $K$ of some dual Banach space $E^*$ and an homeomorphism $H: (X, \tau) \longrightarrow (K,w^*)$, satisfying the following property: for every proper lower semicontinuous function $f : (X,\tau) \longrightarrow  \R \cup \lbrace +\infty \rbrace $, there exists a proper convex weak$^*$ lower semicontinuous function $F: (\overline{\textnormal{co}}^{w^*}(K), w^*) \longrightarrow  \R \cup \lbrace +\infty \rbrace $, such that the following diagramm commutes 

\[
  \xymatrix{
    (X, \tau) \ar[r]^{H} \ar[rd]_{f}  &  (K,w^*) \ar[d]^{F_{|K}} \\
     & \R \cup \lbrace +\infty \rbrace } 
 \]
and 
$$\textnormal{argmin}(f)=H^{-1}(\textnormal{argmin}(F_{|K}))= H^{-1}(\textnormal{argmin}(F)\cap K).$$


\end{Cor}
 \begin{proof} By applying Proposition \ref{Homeo} with $Y=C(X)$, we get that the compact Hausdorff space $(X, \tau)$ is homeomorphic to the weak$^*$ compact subset $(\Delta_Y(X),w^*)$ of $(C(X))^*$. Thus, the conclusion follows from Lemma \ref{Prop1} by setting $K:=\Delta_Y(X) (=\overline{\Delta_Y(X)}^{w^*})$ and $F:=\mathcal{F}(f)$. 
 
 \end{proof}
\section{Duality and inf-convolution.}
In this section, we give the proof of Theorem \ref{inf-conv} below. This theorem has a know analogous in the classical Fenchel duality. In all this section, we assume that $X$ is a normal Hausdorff space and $Y=C_b(X)$. Thus the hypothesis $(H)$ is satisfied for the pair $(X,Y)$ by the Urysohn's lemma (see Exemple \ref{Exemp1}). Recall that $$A_Y(f):=\lbrace \varphi\in Y/ \varphi \leq f\rbrace.$$ Recall that the inf-convolution of two functions $k$ and $l$ defined on vector space $Z$ is defined for all $z\in Z$ as follows:
\begin{eqnarray*}
 k\diamond l(z)&:=& \inf_{y\in Z} \lbrace k(y)+l(z-y)\rbrace
\end{eqnarray*} 

\begin{Thm} \label{inf-conv} Let $X$ be a normal Hausdorff space and $Y=C_b(X)$. Let $f, g : X\longrightarrow \R$ be bounded from below lower semicontinuous functions (here $dom(f)=dom(g)=X$). Then $(f+g)^{\times}=f^{\times}\diamond g^{\times}$ on $C_b(X)$.
\end{Thm}

The proof of this theorem will be given after some preliminary results. 

\begin{Lem} \label{lem1} Let $X$ be a normal Hausdorff space and $Y=C_b(X)$. Then, we have that $A_Y(f+g)=A_Y(f) + A_Y(g)$ for all $f,g: X \longrightarrow \R$ bounded from below and lower semicontinuous (here $dom(f)=dom(g)=X$). 
\end{Lem}
\begin{proof} Clearly we have that $A_Y(f) + A_Y(g) \subseteq A_Y(f+g)$. Let us prove the converse. Indeed, let $\varphi \in A_Y(f+g)$. Then we have $\varphi \leq f+g$. In other words, $\varphi -g \leq f$, with $\varphi -g$ uper semicontinuous and $f$ lower semicontinuous. Since $X$ is a normal space, using the insertion theorem [Theorem 1. \cite{GS}], there exists a continuous function $\psi$ on $X$ such that $\varphi -g \leq \psi \leq f$. Thus, we have $\psi \leq f$ i.e. $\psi\in A_Y(f)$ and $\varphi- \psi \leq g$ i.e.  $\varphi- \psi\in A_Y(g)$ with $\varphi =\psi +(\varphi-\psi)$. Hence $\varphi\in A_Y(f) + A_Y(g)$.

\end{proof}
\begin{Lem} \label{lem2}  Let $X$ be a normal Hausdorff space and $Y=C_b(X)$. Let $f : X\longrightarrow \R \cup \lbrace +\infty \rbrace$ be a bounded from below lower semicontinuous function with $dom(f)\neq \emptyset$. Then, for all $\xi \in C_b(X)$, we have that $$f^{\times}(\xi)=\inf_{\varphi\in A_Y(f)} \varphi^{\times}(\xi).$$ 
\end{Lem}
\begin{proof} It suffices to shows that $f^{\times}(0)=\inf_{\varphi\in A_Y(f)} \varphi^{\times}(0)$, since $f^{\times}(\xi)=(f-\xi)^{\times}(0)$ and $A_Y(f-\xi)=A_Y(f) -\xi$, for all $\xi\in C_b(X)$. From the part $(4)$ in Lemma \ref{Prop1}, we have that $\varphi^{\times}(0):=-\inf_{x\in X}\varphi(x)=-\min_{Q\in \overline{\textnormal{co}}^{w^*}(\Delta_Y(X))} \mathcal{F}(\varphi)(Q)$, for all $\varphi\in C_b(X)$. Thanks to the part $(3)$ of Lemma \ref{Prop1}, we get that $\varphi^{\times}(0)=-\min_{Q\in \overline{\textnormal{co}}^{w^*}(\Delta_Y(X))} \hat{\varphi}(Q)$, for all $\varphi\in C_b(X)$. Thus, we have that 

\begin{eqnarray} \label{Eq1}
\inf_{\varphi\in A_Y(f)} \varphi^{\times}(0)&=&\inf_{\varphi\in A_Y(f)}(-\min_{Q\in \overline{\textnormal{co}}^{w^*}(\Delta_Y(X))} \hat{\varphi}(Q))\nonumber \\
                                         &=& -\sup_{\varphi\in A_Y(f)} \min_{Q\in \overline{\textnormal{co}}^{w^*}(\Delta_Y(X))} \hat{\varphi}(Q) .
\end{eqnarray}
Applying the minimax theorem [Corollary 2., \cite{FT}] to $\langle \cdot, \cdot \rangle: \overline{\textnormal{co}}^{w^*}(\Delta_Y(X)) \times A_Y(f) \longrightarrow \R$ defined by $\langle Q,\varphi \rangle:=\hat{\varphi}(Q)=Q(\varphi)$, we have that $$\sup_{\varphi\in A_Y(f)} \min_{Q\in \overline{\textnormal{co}}^{w^*}(\Delta_Y(X))} \hat{\varphi}(Q)=\min_{Q\in \overline{\textnormal{co}}^{w^*}(\Delta_Y(X))} \sup_{\varphi\in A_Y(f)}\hat{\varphi}(Q).$$ Hence, using (\ref{Eq1}) we obtain that 
\begin{eqnarray} \label{Eq2}
\inf_{\varphi\in A_Y(f)} \varphi^{\times}(0)=-\min_{Q\in \overline{\textnormal{co}}^{w^*}(\Delta_Y(X))} \sup_{\varphi\in A_Y(f)}\hat{\varphi}(Q).
\end{eqnarray}
It is easy to see that $Q(1)=1$ for all $Q\in \overline{\textnormal{co}}^{w^*}(\Delta_Y(X))$. Let $\varphi_0$ be the constant function defined by $\varphi_0(x):=\inf_X f$ for all $x\in X$. We have that $\varphi_0\leq f$ and so, $\sup_{\varphi\in A_Y(f)}\hat{\varphi}(Q)\geq Q(\varphi_0)=\varphi_0$ for all $Q\in \overline{\textnormal{co}}^{w^*}(\Delta_Y(X))$. This implies that $\min_{Q\in \overline{\textnormal{co}}^{w^*}(\Delta_Y(X))}\sup_{\varphi\in A_Y(f)}\hat{\varphi}(Q)\geq \varphi_0$. Thus, from (\ref{Eq2}) we obtain that $\inf_{\varphi\in A_Y(f)} \varphi^{\times}(0)\leq -\varphi_0=-\inf_{x\in X}f:=f^{\times}(0)$. Now, to see the converse, since $\Delta(X)\subset \overline{\textnormal{co}}^{w^*}(\Delta_Y(X))$, we have that $$\min_{Q\in \overline{\textnormal{co}}^{w^*}(\Delta_Y(X))} \sup_{\varphi\in A_Y(f)}\hat{\varphi}(Q)\leq \inf_{x\in X} \sup_{\varphi\in A_Y(f)}\varphi(x)\leq \inf_{x\in X} f(x).$$ 
Since $\inf_{x\in X} f(x):=-f^{\times}(0)$, using (\ref{Eq2}), we have that $f^{\times}(0)\leq \inf_{\varphi\in A_Y(f)} \varphi^{\times}(0)$. 
\end{proof}

\vskip5mm

\begin{proof}[Proof of Theorem \ref{inf-conv}] It is easy to see that for all $\varphi, \psi \in C_b(X)$, we have 
$$f^{\times}(\psi) +g^{\times}(\varphi-\psi)\geq (f+g)^{\times}(\varphi).$$
Taking the infinimum over $\psi\in C_b(X)$, we get that $(f+g)^{\times}\leq f^{\times}\diamond g^{\times}$. 
It is easy to verify the following Claim.  

\noindent {\it Claim.} If $\varphi, \psi \in C_b(X)$, then $(\varphi + \psi)^{\times} = \varphi^{\times} \diamond \psi^{\times}$ on $C_b(X)$. 

\vskip5mm

\noindent Now, let $\theta, \xi \in C_b(X)$. From the Lemma \ref{lem2} we have that 
\begin{eqnarray*}
f^{\times}(\xi)+g^{\times}(\theta-\xi)&=&\inf_{\varphi\in A_Y(f)} \varphi^{\times}(\xi)+\inf_{\psi\in A_Y(g)} \psi^{\times}(\theta-\xi)\\
                                    &=& \inf_{(\varphi,\psi)\in A_Y(f)\times A_Y(g)}\lbrace\varphi^{\times}(\xi)+\psi^{\times}(\theta-\xi)\rbrace
\end{eqnarray*}
By taking the infinimum over $\xi$ in the above formula, we obtain
\begin{eqnarray*}
f^{\times}\diamond g^{\times}(\theta)&=& \inf_{(\varphi,\psi)\in A_Y(f)\times A_Y(g)}\lbrace\varphi^{\times}\diamond \psi^{\times}(\theta).\rbrace
\end{eqnarray*}
Using the {\it Claim.} we have 
\begin{eqnarray*}
f^{\times}\diamond g^{\times}(\theta)&=& \inf_{(\varphi,\psi)\in A_Y(f)\times A_Y(g)}\lbrace (\varphi+\psi)^{\times}(\theta)\rbrace.
\end{eqnarray*}
From Lemma \ref{lem1} we have
\begin{eqnarray*}
f^{\times}\diamond g^{\times}(\theta)&=& \inf_{\mu\in A_Y(f+g)}\mu^{\times}(\theta).
\end{eqnarray*}

Using again Lemma \ref{lem2}, we get that $f^{\times}\diamond g^{\times}(\theta)=(f+g)^{\times}(\theta)$ for all $\theta \in C_b(X)$.
 \end{proof}

\section{Proof of Theorem \ref{Iso-cone}.} 
For the proof of Theorem \ref{Iso-cone}, we also need the following propositions.
 \begin{Prop} \label{Delta-convex} Let $X$ be a topological space, $Y=C_b(X)$ and $\emptyset\neq A\subset C_b(X)$. Then, $A$ is $\Delta$-set if and only if, there exists a bounded from below lower semicontinuous function $f: X\longrightarrow \R$ such that $A=A_Y(f):=\lbrace \psi\in C_b(X)/ \psi \leq f\rbrace$.
\end{Prop}
\begin{proof} Let us prove the "only if part". Since $A$ is $\Delta$-set, there exists real numbers $\lambda_x \in \R$, for all $x\in X$, such that $A=\cap_{x\in X}\lbrace \varphi \in C_b(X): \varphi(x)\leq \lambda_x \rbrace$. Let us set $f(x):= \sup_{\psi \in A } \psi(x)$, for all $x\in X$. Thus, we have $f(x)\leq \lambda_x< +\infty$ for all $x\in X$. It follows that $A_Y(f)\subset A$, that $dom(f)=\R$ and that the function $f$ is lower semicontinuous as supremum of continuous function. It follows also that $f$ is bounded from below, since there exists a bounded continuous function $\varphi \in A$ such that $-\infty < \inf_X \varphi \leq \varphi \leq f$. On the other hand, if $\varphi\in A$, then for all $x\in X$ we have $\varphi(x) \leq \sup_{\psi \in A } \psi(x):=f(x)$. This shows that $\varphi\in A_Y(f)$ and so that $A\subset A_Y(f)$. Hence $A=A_Y(f)$.  The "if part" is clear.

\end{proof} 

 \begin{Prop} \label{Prop2} Let $X$ be a normal Hausdorff space and $Y=C_b(X)$. Let $f : X\longrightarrow \R \cup \lbrace +\infty \rbrace$ be a bounded from below lower semicontinuous function with $dom(f)\neq \emptyset$. Then, for all $Q\in \overline{\textnormal{co}}^{w^*}(\Delta_Y(X))$, we have that 

$$\mathcal{F}(f)(Q):=(f^{\times})^*(Q)=\sup_{\varphi \in A_Y(f)} \langle Q,\varphi \rangle:=\sigma_{A_Y(f)}(Q).$$
\end{Prop}
\begin{proof} From Lemma \ref{lem2}, we have that $f^{\times}(\xi)=\inf_{\varphi\in A_Y(f)} \varphi^{\times}(\xi)$, for all $\xi \in C_b(X)$. So, by applying the Fenchel conjugacy to $f^{\times}$, we get $(f^{\times})^*(Q)=\sup_{\varphi\in A_Y(f)} (\varphi^{\times})^*(Q)$ for all $Q\in (C_b(X))^*$. Hence, $\mathcal{F}(f)(Q)=\sup_{\varphi\in A_Y(f)} \mathcal{F}(\varphi)(Q)$ for all $Q\in (C_b(X))^*$. Using the part $(3)$ of Lemma \ref{Prop1}, we obtain that  $\mathcal{F}(f)(Q)=\sup_{\varphi\in A_Y(f)} \langle Q, \varphi \rangle$, for all $Q\in \overline{\textnormal{co}}^{w^*}(\Delta_Y(X))$.

\end{proof}

\vskip5mm

\begin{proof}[Proof of Theorem \ref{Iso-cone}] First, we define $\mathcal{T}$ as follows: for all $f\in SCI(X)$, $$\mathcal{T}(f):=\mathcal{F}(f)_{|\overline{\textnormal{co}}^{w^*}(\Delta_Y(X))},$$ the restriction of $\mathcal{F}(f)$ to ${\overline{\textnormal{co}}^{w^*}(\Delta_Y(X))}$. From Proposition \ref{Prop2}, we have that $\mathcal{T}(f) \in \Gamma(\overline{\textnormal{co}}^{w^*}(\Delta(X)))$ for all $f\in SCI(X)$. Let us prove that $\mathcal{T}: SCI(X)\longrightarrow \Gamma(\overline{\textnormal{co}}^{w^*}(\Delta(X)))$ is a bijective map. Indeed, using the part $(1)$ of Lemma \ref{Prop1}, we get that $\mathcal{T}$ is one to one. To see that $\mathcal{T}$ is onto, let $g \in \Gamma(\overline{\textnormal{co}}^{w^*}(\Delta(X)))$, there exists a $\Delta$-set $A$ such that $g= (\sigma_A)_{|\overline{\textnormal{co}}^{w^*}(\Delta(X))}$. Using Proposition \ref{Delta-convex}, there exist a bounded from below lower semicontinuous function $f: X\longrightarrow \R$ such that $A=A_Y(f):=\lbrace \psi\in C_b(X)/ \psi \leq f\rbrace$. Thus, by using Proposition \ref{Prop2}, we get that $g=\mathcal{T}(f)$ i.e. $\mathcal{T}$ is onto. Hence, $\mathcal{T}$ is a bijective map. Now, we prove that for all $f, g \in SCI(X)$ and all $\alpha,\beta \in \R^+$, we have $$\mathcal{T}(\alpha f+\beta g)=\alpha\mathcal{T}(f)+ \beta \mathcal{T}(g).$$

Indeed, let $\alpha\in \R^+$. If $\alpha=0$, then from the part $(2)$ of Lemma \ref{Prop1}, we have that $\mathcal{T}(0)= \mathcal{F}(0)=0$ on $\overline{\textnormal{co}}^{w^*}(\Delta_Y(X))$. If $\alpha\neq 0$, it is easy to see that $(\alpha f)^{\times}(\varphi)=\alpha f^{\times}(\frac{\varphi}{\alpha})$ for all $f\in SCI(X)$. Thus, $((\alpha f)^{\times})^*=\alpha (f^{\times})^*$ which implies that $\mathcal{T}(\alpha f)=\alpha \mathcal{T}(f)$ for all $f\in SCI(X)$. On the other hand, if $f,g \in SCI(X)$, then by applying Theorem \ref{inf-conv}, we get that $(f+g)^{\times}=f^{\times}\diamond g^{\times}$. Hence by the properties of the Fenchel conjugacy, $((f+g)^{\times})^*=(f^{\times}\diamond g^{\times})^*=(f^{\times})^* + (g^{\times})^*$. In other words, we have that $\mathcal{T}(f+g)=\mathcal{T}(f)+\mathcal{T}(g)$. Thus, $\mathcal{T}(\alpha f + \beta g)=\alpha \mathcal{T}(f)+ \beta \mathcal{T}(g)$, for all $f, g \in SCI(X)$ and all $\alpha, \beta \in \R^+$. It follows from this formula that the set $\Gamma(\overline{\textnormal{co}}^{w^*}(\Delta(X)))$ is a convex cone and that $\mathcal{T}$ is an isomorphism of convex cone.

\vskip5mm

The parts $(1)$, $(3)$ and $(4)$ of the theorem, follows repectively from the parts $(1)$, $(3)$ and $(4)$ of Lemma \ref{Prop1}. Now, we prove the part $(2)$ of the theorem. Let $f,g\in SCI(X)$. If $f\leq g$, then we see easily from the definition that $g^{\times} \leq f^{\times}$. Also from the definition of the Fenchel conjugacy we get that $(f^{\times})^* \leq (g^{\times})^*$ which implies that $\mathcal{T}(f)\leq \mathcal{T}(g)$. Now, suppose that $\mathcal{T}(f)\leq \mathcal{T}(g)$. Since $dom(\mathcal{F}(f))\subset \overline{\textnormal{co}}^{w^*}(\Delta_Y(X))$ (see the part $(1)$ of Lemma \ref{Prop1}), we also have that $\mathcal{F}(f) \leq \mathcal{F}(g)$ i.e. $(f^{\times})^*\leq (g^{\times})^*$. This implies that $(g^{\times})^{**}\leq (f^{\times})^{**}$. Since $f^{\times}$ and $g^{\times}$ are convex and $1$-Lipschitz continuous functions, using the classical Fenchel-Moreau theorem, we obtain that $g^{\times}\leq f^{\times}$. This implies that $f^{\times\times}\leq g^{\times\times}$ and so by applying Theorem \ref{dual1}, we get that $f\leq g$.    

\end{proof}

\bibliographystyle{amsplain}

\end{document}